\documentclass{amsart}
\usepackage[english]{babel}
\usepackage[a4paper,top=2cm,bottom=2cm,left=3cm,right=3cm,marginparwidth=2cm]{geometry}
\usepackage{tikz}
\usetikzlibrary{positioning}
\usepackage{pdflscape}
\usepackage{tikz-cd}
\usepackage{comment}
\usepackage{biblatex}
\usepackage{mathtools}
\usepackage{amsmath}
\usepackage{amsfonts}
\usepackage{amssymb}
\usepackage{amsthm}
\usepackage{amsopn}
\usepackage{mathrsfs}
\usepackage{graphics}
\renewbibmacro{in:}{\addcomma\addspace}
\usepackage{amsaddr,amsmath, amsfonts, amssymb, graphicx, tikz-cd,tikz,enumitem,mathtools,hyperref, mathabx, mathrsfs, mathabx, cleveref, mdframed, adjustbox,csquotes,amssymb,amsthm,todonotes}
\usetikzlibrary{calc,shapes}

\newtheorem{theorem}{Theorem}[section]
\newtheorem{lemma}[theorem]{Lemma}
\newtheorem{corollary}[theorem]{Corollary}
\newtheorem{proposition}[theorem]{Proposition}

\theoremstyle{definition}
\newtheorem{definition}[theorem]{Definition}
\newtheorem{example}[theorem]{Example}

\theoremstyle{remark}
\newtheorem{remark}[theorem]{Remark}

\renewcommand{\r}{\mathfrak{r}}

\newcommand{\C}[0]{\mathbb{C}}

\newcommand{\cal}[1]{\mathcal{#1}}

\newcommand{\Z}[0]{\mathbb{Z}}

\newcommand{\arr}[1]{\overset{#1}{\rightarrow}}
\newcommand{\inner}[2]{\langle #1, #2 \rangle}

\newcommand{\id}[0]{\operatorname{id}}
\newcommand{\FC}[0]{\mathbb{F} C^*}
\renewcommand{\r}{\frk{r}}
\newcommand{\s}{\frk{s}}

\renewcommand{\k}[0]{\mathbb{K}}
\newcommand{\order}[0]{\operatorname{ord}}
\newcommand{\diag}[0]{\operatorname{diag}}
\newcommand{\K}[0]{\mathbb{K}}
\newcommand{\trace}[1]{\operatorname{trace}(#1)}

\newcommand{\scr}[1]{\mathscr{#1}}
\newcommand{\frk}[1]{\mathfrak{#1}}

\begin{document}

\author[Graham et al.]{Joshua Graham $^{a}$, Rishabh Goswami $^{b}$ and Jason Palin $^{c}$}
\address{$^{a}$School of Mathematics and Statistics, University of New South Wales, Sydney, Australia,\\
$^{b}$Department of Mathematics, North-Eastern Hill University, Shillong, India,\\
$^{c}$Department of Mathematics and Statistics, McMaster University, Hamilton, Canada.\\}

\email{joshua.graham@unsw.edu.au, rishabhgoswami.math@gmail.com and palinj@mcmaster.ca}

\title{Leavitt Path Algebras of Quantum Quivers}

\begin{abstract}
Adapting a recent work of Brannan et al., on extending graph $C^*$-algebras to Quantum graphs, we introduce "Quantum Quivers" as an analogue of quivers where the edge and vertex set has been replaced by a $C^*$-algebra and the maps between the sets by $*$-homomorphisms. Additionally, we develop the theory around these structures and construct a notion of Leavitt path algebra over them and also compute the monoid of finitely generated projective modules over this class of algebras. 

\end{abstract}
\maketitle

\section{Introduction}

Quantum graphs were first introduced in \cite{Weaver} as an application of a noncommutative notion of relations in the setting of von Neumann algebras. Since then they have been applied and studied largely in the context of quantum communication and noncommutative algebras. Several different definitions of quantum graphs have been proposed since their origin \cite{Musto, kuperberg2012neumann, chirvasitu2022random}. In particular they have been defined in terms of a self-adjoint operator on $L^2(B)$, $B$ a finite dimensional $C^*$-algebra, which is idempotent under pointwise multiplication generalizing the usual adjacency matrix of a graph (which contains only 1's and 0's) (\cite[Definition~5.1]{Musto}). This definition was used in \cite{Voigt} to develop a quantum graph version of Cuntz-Krieger $E$-families, allowing also for analogs of Cuntz-Krieger algebras and graph $C^*$-algebras to be defined over quantum graphs.

There is a disconnect between classical directed graphs and quantum graphs stemming from the fact that quantum graphs lack a good notion of many usual objects in graph theory, like paths and cycles. These have received particular importance in recent years in the context of corresponding graph theoretic properties with properties of algebraic structures associated to graphs, such as in the context of Leavitt path algebras (LPAs for short) \cite{AAS} and particularly in the context of graph monoids \cite{ara2007nonstable} and topological groupoids \cite{rigby}. With these connections in mind, there is clear utility for developing usable graph theoretic notions in the quantum graph setting. 

The aim of this paper is to study a slight variation of quantum graph, which we call \emph{quantum quiver} (QQ for short) which comes from applying ideas in noncommutative geometry to discrete graph theory. The construction of quantum quivers is along the lines of topological quivers introduced in \cite{muhly2005topological}, where the vertex and the edge set is replaced by a locally compact Hausdorff space with the range and source maps between the sets by continuous open maps.

When combined with the result that finite-dimensional $C^*$-algebras are products of matrix algebras, we find that the structures buried within QQs are elegant and in many ways behave similarly to directed graphs. Moreover, we present a diagrammatic way to represent such structures that mirrors the way directed graphs can be drawn, though we do lose a small amount of information doing this (see Diagram \ref{qqdiagram2} for instance). Building on the foundations in \cite{Voigt}, we also define quantum Cuntz-Krieger algebras for QQs and from this define a ring with generators and relations that helps to connect the theory of quantum quivers to Leavitt path algebras. We would like to emphasise that these introduced families of operators and associative algebras do not necessarily obey the conventional structural results. However, building upon ideas introduced in \cite{Bergman1974} we are able to explicitly calculate the monoid of finitely generated projective modules over the aforementioned algebras. This gives geometric information on QQs akin to how graph monoids give information about directed graphs. 

In section 2 of this article we introduce the notion of a quantum quiver, an analogue of quantum graphs which has an associated Cuntz-Krieger family with two distinguished linear maps corresponding to vertices and edges. We dwell on the definition of quantum quivers for the rest of the section, examining the notion of isomorphism between them and the question of classification, introducing a diagrammatic tool for visualizing QQs, and defining analogies of the classical graph-theoretic notions of complete and disconnected graphs. In section 3 we accomplish our original goal of associating LPAs to QQs while in section 4 we compute the monoid of finitely generated projective modules over these algebras.

\section{Quantum quivers}
Our aim a priori is to apply define LPAs for quantum graphs using the framework of the quantum Cuntz-Krieger family (see Definition \ref{def_2_5}). However, the quantum Cuntz-Krieger family consist solely of partial isometries, analogous to edges in the classical case, but we also need a distinguished class of elements analogous to vertices in order to have a version of LPAs associated with quantum graph. Thus, we introduce the notion of a ``Quantum Quiver" which has the notion of both edges and vertices.
\subsection{Definition and Basic Properties}

\begin{definition}
    A \textit{quantum quiver} (QQ) is a collection $\scr{B} = (B_0, B_1, \r,\s)$ where $B_0,B_1$ are finite dimensional complex $C^*$-algebras and $\r,\s: B_0 \to B_1$ are unit-preserving $*$-homomorphisms.
\end{definition}

We may turn any discrete directed graph into a quantum quiver in the following way: given two finite sets $E_0,E_1$ with set functions $r,s: E_1 \to E_0$ we have induced algebra homomorphisms $r^*,s^*: C(E_0) \to C(E_1)$ (where $C(E_0)$ denotes the $C^*$-algebra of complex functions in $E_0$). Thus $(C(E_0),C(E_1),r^*,s^*)$ is a quantum quiver with the constituent $C^*$-algebras commutative. This assignment is actually a bijection by way of the following.

\begin{lemma}
    Suppose $\scr{B} = (B_0,B_1,\r,\s)$ is a QQ with $B_0,B_1$ commutative. Then there is a directed graph $(E_0,E_1,r,s)$ such that $(B_0,B_1,\r,\s)$ is isomorphic to $(C(E_0),C(E_1),r^*,s^*)$.
\end{lemma}

\begin{proof}
    It is well-known that there is a one-to-one correspondence between locally compact Hausdorff spaces to commutative $C^*$-algebras via sending a space $M$ to the algebra of complex continuous functions on $M$. This restricts to a one-to-one correspondence from finite discrete spaces to finite dimensional commutative $C^*$-algebras (i.e., complex algebras of the form $\C^n$ for some integer $n$). Assume without loss of generality that $B_i$ is given by $\C^{n_i}$ for some integers $n_i$, $i = 0,1$. Hence we may further assume that for each $i$, $B_i \cong C(E_i)$ for some finite sets $E_i$ of size $n_i$. Since $\r,\s$ are unit-preserving $\C$-algebra homomorphisms, it follows in the framework we've said above that both $\r,\s$ must be induced by set-theoretic maps $r,s: E_1 \to E_0$ (respectively). It follows by definition that $(E_0, E_1,r,s)$ is a directed graph and that $(B_0,B_1,\r,\s)$ is isomorphic $(C(E_0), C(E_1),r,s)$.
\end{proof}

It is also helpful for what follows to have a notion of morphisms for quantum quivers. The categorically-minded reader will find the following definition obvious.

\begin{definition}
    A \emph{morphism of QQs} $\scr{B} = (B_0,B_1,\r_B, \s_B) \arr{\Phi} \scr{C} = (C_0, C_1, \r_C, \s_C)$ is a pair of $*$-homomorphisms $\Phi_i: B_i \to C_i$ for $i=0,1$, such that the squares
    \begin{equation}
        \begin{tikzcd}
            B_0 \arrow[r, "\Phi_0"] \arrow[d, "\r_B"] & C_0 \arrow[d, "\r_C"] \\
            B_1 \arrow[r, "\Phi_1"] & C_1
        \end{tikzcd}, \begin{tikzcd}
            B_0 \arrow[r, "\Phi_0"] \arrow[d, "\s_B"] & C_0 \arrow[d, "\s_C"] \\
            B_1 \arrow[r, "\Phi_1"] & C_1
        \end{tikzcd}
    \end{equation}
    commute. A morphism $\Phi$ of the above form is said to be an \emph{isomorphism} if $\Phi_0,\Phi_1$ are $*$-isomorphisms. In which case we say $\scr{B}, \scr{C}$ are isomorphic.
\end{definition}

It is well known that a finite dimensional $C^*$-algebra can be written as a direct product of matrix algebras over $\C$ (see \cite{Murphy1990}, Theorem 6.3.8). We will hence tacitly assume going forward $B_0, B_1$ have decompositions
$$\prod_{v \in I_0} M_{n_v}(\C),$$
and
$$\prod_{\alpha \in I_1} M_{m_\alpha}(\C),$$
for some finite indexing sets $I_0, I_1$, and integers $n_v, m_\alpha$ for each $v \in I_0, \alpha \in I_1$.

\begin{definition}
    Call a QQ $\scr{B} = (B_0, B_1, \r, \s)$ \textit{complete} if $\r^* \s: B_0 \to B_0$ sends $1_v \in M_{n_v}(\C) \subset B_0$ to $1 \in B_0$. Call $\scr{B}$ \textit{disconnected} if $B_1 = 0$.
\end{definition}

For any number of vertices, one can always construct a set of edges between these vertices such that the resulting quiver is complete (recall a quiver $(E_0, E_1, r,s)$ is complete if for all $x,y \in E_0$, there is precisely one $e \in E_1$ such that $r(e) = x, s(e) = y$). Moreover, this complete quiver is unique up to isomorphism. We hence ask the two following questions;
\begin{enumerate}
    \item \textit{Given a finite dimensional $C^*$-algebra $B_0$, is it possible to construct another finite dimensional $C^*$-algebra $B_1$ and unital $*$-homomorphisms $\r,\s: B_0 \to B_1$ such that $\scr{B} = (B_0, B_1, \r,\s)$ is complete?}
    \item \textit{Given complete QQs $\scr{C} = (C_0, C_1, \r_C, \s_C)$ and $\scr{B} = (B_0,  B_1, \r_B, \s_B$) such that $B_0 \cong C_0$, is it true that $B_1 \cong C_1$?}
\end{enumerate}

Unfortunately, the answer to both questions is no. The following lemma illustrates the reasoning behind our first answer: completeness of $\scr{B}$ puts constraints on the structure of $B_0$. The proof of the following references a theorem that will be proved later.

\begin{lemma} \label{lem_3_8}
    Let $\scr{B} = (B_0, B_1, \r,\s)$ be a complete QQ and let $v \in I_0$ be an index of one of the matrix algebra summands of $B_0$ (with corresponding matrix algebra $M_{n_v}(\C)$). Then $\sum_{w \in I_0} n_w$ is a multiple of $n_v$.
\end{lemma}
\begin{proof}
    By Theorem \ref{skolem_noether_quantum}, the image of $1_v \in B_0$ under $\s$ will be a diagonal matrix $D$ of 1's and 0's. Moreover, the number of 1's along the diagonal of $D$ will be a multiple of $n_v$. By definition of the adjoint, we have $\trace{D \r(a)} = \trace{a}$ since $\r^*(D) = 1$ for all $a \in B_0$. In particular for $a = 1_{B_0}$, we have $\trace{D} = \trace{1_{B_0}} = \sum_{w \in I_0}n_w$. The result now follows.
\end{proof}

It is currently unclear if $B_0$ satisfying Lemma \ref{lem_3_8} is sufficient to construct a complete QQ with it, although we suspect the answer is yes.

Concerning the second question, we consider the case where $B_0 = \C^2$. We can construct two complete QQs from $B_0$ which are not mutually isomorphic, as illustrated below. Let $B_1 = \C^4$, with $\r: B_0 \to B_1: (\lambda_1, \lambda_2) \mapsto (\lambda_1, \lambda_2, \lambda_1, \lambda_2)$ and $\s: B_0 \to B_1: (\lambda_1, \lambda_2) \mapsto (\lambda_1, \lambda_1, \lambda_2, \lambda_2)$. It straightforward to verify that $(\C^2, \C^4, \r,\s)$ is complete. Consider then, the diagonal embedding $\iota: \C^4 \mapsto M_4(\C)$. It can also be checked that $(\C^2, M_4(\C), \iota \r, \iota \s)$ is complete.

\vspace{3mm}

For a finite dimensional $C^*$-algebra $B$ equipped with a faithful state $\psi:B \to \C$, we have an induced Hilbert space structure on $B$ with inner product $\inner x y := \psi(x^*y)$. We denote this Hilbert space by $L^2(B) := L^2(B, \psi)$. With respect to this structure multiplication (seen as a bounded linear operator) $m: L^2(B \otimes B) \cong L^2(B) \otimes L^2(B) \to L^2(B): x \otimes y \mapsto xy$ has a Hilbert space adjoint which we'll call \textit{comultiplication} $m^*: L^2(B) \to L^2(B \otimes B)$. We call $\psi$ a \textit{$\delta$-form} if $mm^* = \delta^2 \id_B$ for some $\delta > 0$.

The following two definitions are from \cite{Voigt}, and Definition \ref{def_2_5} is original to that article.

\begin{definition}
    A \textit{quantum graph} is a collection $\scr{G} = (B,\psi,A)$ where $B$ is a finite dimensional $C^*$-algebra, $\psi: B \to \C$ is a $\delta$-form, and $A: L^2(B) \to L^2(B)$ is a linear operator such that $m(A \otimes A)m^* = \delta^2 A$.
\end{definition}

\begin{definition} \label{def_2_5}
    Given a quantum graph $\scr{G} = \{B,\psi,A\}$, a \textit{quantum Cuntz-Krieger $\scr{G}$-family} is a pair $(s,D)$ where $D$ is a $C^*$-algebra and $s:B \to D$ is a linear map (not necessarily a $*$-homomorphism) such that
    \begin{enumerate}
        \item$\mu(s^* \otimes s)m^* = \mu(s \otimes s^*)m^* A$,
        \item $\mu(1 \otimes \mu)(s \otimes s^* \otimes s)(1 \otimes m^*)m^* = s$,
    \end{enumerate}
    where $\mu: D \otimes D \to D$ denotes multiplication and $s^*$ denotes the linear map $b \mapsto s(b^*)^*$. Define the \textit{free quantum graph $C^*$-algebra} $\FC(\scr{G})$ to be the universal $C^*$-algebra with respect to these quantum Cuntz-Krieger $\scr{G}$-families.
\end{definition}

\begin{lemma} \label{lem_3_2}
    Let $(B_0,B_1,\r,\s)$ be a QQ. For a fixed $\delta > 0$, suppose further that $\psi_0, \psi_1$ are $\delta$-forms on $B_0, B_1$ respectively. Let $m_1: B_1 \otimes B_1 \to B_1$ denote multiplication and $m_1^*: B_1 \to B_1 \otimes B_1$ denote the adjoint with respect to the Hilbert space structure on $B_1$ induced by $\psi_1$. Similarly, let $\r^*: B_1 \to B_0$ denote the adjoint of $\r$. Then
    $$m_1(\s \r^* \otimes \s \r^*)m_1^* = \delta^2 \s \r^*.$$
    In particular, $\scr{G} = (B_1, \psi_1, \s \r^*)$ is a quantum graph.
\end{lemma}

\begin{proof}
    This follows from $\s$ preserving multiplication (i.e., $\s m_0 = m_1(\s \otimes \s)$) and likewise $\r^*$ preserving comultiplication;
    \begin{align*}
        m_1(\s \r^* \otimes \s \r^*)m_1^* & = m_1(\s \otimes \s)(\r^* \otimes \r^*)m_1^* \\
        & = \s m_0 m_0^* \r^* \\
        & = \delta^2 \s \r^*.
    \end{align*}
    Recall that $m_0 m_0^* = \delta^2 \id$.
\end{proof}

It is natural to ask if a classical quantum graph $(B, \psi, A)$ can be extended to a QQ in our sense. In general the answer is no; even in the commutative setting it is easy to find counterexamples. It is unclear at present what further conditions $A$ should satisfy to permit the existence of such an extension.

That being said we will no longer require $B_0, B_1$ come equipped with $\delta$-forms, with these forms playing no further roles in our constructions or results and we will focus entirely on QQs. The adjoint maps are defined are defined with respect to the trace inner-product $\inner{x}{y} = \operatorname{trace}(x^*y)$ on finite dimensional $C^*$-algebras. Recall that finite dimensional $C^*$-algebras are direct sums of matrix algebras with $*$-operation being conjugate transpose.

We'll introduce a modification of free quantum Cuntz-Krieger algebras for QQs. The following definition acts as a midway definition between free quantum Cuntz-Krieger algebras for quantum graphs and graph $C^*$-algebras for directed graphs. 

\begin{definition} \label{defn_free_quantum_B_family}
    Let $\scr{B} = (B_0, B_1,\r,\s)$ be a QQ. A \textit{free quantum Cuntz-Krieger} $\scr{B}$-family is a collection $(\phi_0, \phi_1, D)$ consisting of a $C^*$-algebra $D$ and linear maps $\phi_i:B_i \to D$, $i=0,1$ such that
    \begin{enumerate}
        \item $\phi_0 = \phi_0^* = \mu(\phi_0 \otimes \phi_0)m_0^*$ (recall $\phi_i^*: b \mapsto \phi_i(b^*)^*$),
        \item $\mu(\phi_0 \s^* \otimes \phi_1) m_1^* = \phi_1 = \mu(\phi_1 \otimes \phi_0 \r^*) m_1^*$,
        \item $\mu(\phi_1^* \otimes \phi_1)m_1^* = \phi_0 \r^*$, and
        \item $\mu(\phi_1 \otimes \phi_1^*)m_1^*\s = \phi_0$,
    \end{enumerate}
    where $\mu: D \otimes D \to D$ denotes multiplication. The \textit{free quantum graph $C^*$-algebra} $\FC(\scr{B})$ is the universal $C^*$-algebra generated with respect to all quantum Cuntz-Krieger $\scr{B}$-families.
\end{definition}
The above definition immediately implies the following remark,
\begin{remark}
   $\scr{B}$ yields the same notion as the free Cuntz-Krieger family and free graph $C^*$-algebra.
\end{remark}

\begin{lemma}
    Suppose $\scr{B} = (B_0, B_1, \r,\s)$ is a QQ, and that $B_0, B_1$ are both equipped with $\delta$-forms $\psi_0, \psi_1$ (respectively) so that $\scr{G} = (B_1, \psi_1, \s\r^*)$ is a quantum graph. Let $(\phi_0, \phi_1, D)$ be a free quantum Cuntz-Krieger $\scr{B}$-family. Then $(\phi_1, D)$ is a quantum Cuntz-Krieger $\scr{G}$-family.
\end{lemma}

\begin{proof}
    We first verify that condition (2) of Definition \ref{def_2_5} is satisfied. Following the definition, we have
    \begin{align*}
        \mu_1(1 \otimes \mu_1)(\phi_1 \otimes \phi_1^* \otimes \phi_1)(1 \otimes m_1^*)m_1^* & = \mu(\phi_1 \otimes \mu(\phi_1^* \otimes \phi_1)m_1^*)m_1^* \\
        & = \mu(\phi_1 \otimes \phi_0 \r^*)m_1^* \\
        & = \phi_1.
    \end{align*}
    To see that condition (1) is satisfied, we have
    \begin{align*}
        \mu(\phi_1^* \otimes \phi_1)m_1^* & = \phi_0 \r^* \\
        & = \mu(\phi_1 \otimes \phi_1^*)m_1^* \s \r^*,
    \end{align*}
    by (3) and (4) of Definition \ref{defn_free_quantum_B_family}.
\end{proof}

\subsection{Classification of Quantum Quivers}
Before giving a classification result on QQs, we briefly turn to establishing the notation we will use for representing $*$-homomorphisms of finite dimensional $C^*$-algebras and introducing several important results. Recall that for a finite dimensional (complex) $C^*$-algebra $B$, it can be written as
$$\prod_{k=1}^A M_{n_k}(\C)$$
with term-wise addition and multiplication, and $*$-operation given by Hermitian adjoint. Clearly as a complex vector space, $B$ has basis $\{e_{ij}^k: 1 \leq k \leq A, 1 \leq i,j \leq n_k\}$ where $e_{ij}^k \in M_{n_k}(\C) \subset B$ is the matrix with a 1 in its $(ij)$-th entry and 0's everywhere else. Multiplication is defined by $e_{ij}^k e_{uv}^w = \delta_{kw} \delta_{ju} e_{iv}^k$ (where $\delta$ here denotes the Kronecker Delta) and the $*$-operation is defined by $(e_{ij}^k)^* = e_{ji}^k$. The following is then readily observed.

\begin{lemma}
    As we've defined it above, $B$ can be given the structure of a finite dimensional Hilbert space with inner product given by $\inner{x}{y} := \operatorname{trace}(x^*y)$. With respect to this inner product, the $e_{ij}^k$ form an orthonormal basis of $B$.
\end{lemma}

We briefly make some observations about $*$-homomorphisms of $C^*$-algebras, beginning our discussion with the following well-known fact.

\begin{lemma}
    Let $\K$ be a field. Suppose there is a (not necessarily unit preserving) ring homomorphism $\varphi: M_n(\K) \to M_m(\K)$. Then $\varphi$ is either the zero map or is injective. In particular, $\varphi$ is the zero map if $m < n$.
\end{lemma}

\begin{proof}
    This follows since $M_n(\k)$ is simple.
\end{proof}

The following result is a modification of the Skolem-Noether theorem, which will be critical in our classification theorem of quantum quivers.

\begin{theorem} \label{skolem_noether_quantum}
    Let $\k$ be a field. Suppose $A = A_1 \times A_2 \times ... \times A_n$ is a product of simple $\k$-algebras and suppose $f,g :A \to M_m(\k)$ are (unital) ring homomorphisms such that for each $i$, we have $\dim_k(A_i \otimes_A V_f) = \dim_k(A_i \otimes_A V_g)$ where $V_f, V_g$ denote $\k^m$ with the $A$-module structure induced by $f,g$ respectively. Then there is some $G \in GL_m(\k)$ such that for all $a \in A$, $f(a) = G^{-1} g(a) G$.
\end{theorem}

\begin{proof}
    For each $i$, for shorthand we denote $V_{f,i} := A_i \otimes_A V_f$ and $V_{g,i} := A_i \otimes_A V_g$. We can decompose each of these finite-dimensional spaces into their irreducible $A_i$-submodules, which must be isomorphic since $A_i$ is simple. Thus, since $\dim_k(V_{f,i}) = \dim_k(V_{g,i})$, they must be isomorphic as $A_i$-modules. Let such an isomorphism be denoted $\varphi_i: V_{f,i} \to V_{g,i}$. We may treat $V_{f,i}, V_{g,i}$ as $A$-modules via the projection $A \to A_i$, and moreover we have $V_f \cong \bigoplus_{i=1}^n V_{f,i}$ as $A$-modules and similarly $V_g \cong \bigoplus_{i=1}^n V_{g,i}$. We can obtain an $A$-module isomorphism $\varphi: V_f \to V_g$ via direct sum of these $\varphi_i$. We may write such an isomorphism as left-multiplication by an invertible matrix $G \in GL_m(\k)$. Hence, for all $v \in k^m$ and $a \in A$ we have $G f(a) v = g(a) G v$, proving the theorem.
\end{proof}

\begin{remark} \label{rem_4_0_1}
    Using the setup of Theorem \ref{skolem_noether_quantum}, with $f: A_1 \times ... \times A_n \to M_m(\k)$ a unital homomorphism, it can be observed that $\sum_{i=1}^n \dim_\k(A_i \otimes_A V_f) = \dim_\K(\prod_{i=1}^n A_i \times_A V_f) = \dim_\k(A \times_A V_f) = \dim_\K(\K^m) = m$. For our purposes, we're interested in the case where each $A_i$ is a matrix algebra $M_{n_i}(\k)$. In such a case, by the previous theorem we may assume that the map $M_{n_i}(\K) \to M_m(\K)$ is an inclusion of matrix algebras or the zero map. Hence for $\sigma \in M_{n_i}(\k)$ the image of the induced $\k$-linear map $f(\sigma): \k^m \to \k^m$ will be a subspace whose dimension is a multiple of $n_i$. Hence $\dim_\k(M_{n_i}(\k) \otimes_A V_f)$ is a multiple of $n_i$.
\end{remark}

\begin{corollary} \label{skolem_noether_quantum_2}
    Let $B = \prod_{\alpha \in E} M_{m_\alpha}(\k)$ be a finite product of matrix algebras over $\k$, and let $A = A_1 \times A_2 \times ... \times A_n$ be a product of simple $\k$-algebras. Then for (unital) ring homomorphisms $f,g: A \to B$, let $V_{f,\alpha},V_{g, \alpha}$ denote $\k^{m_\alpha}$ with an $A$-module structure via $f,g$ (respectively) followed by projection $B \to M_{m_\alpha}(\k)$. If for all $i=1,...,n$ and $\alpha \in E$, we have $\dim_\k(A_i \otimes_A V_{f,\alpha}) = \dim_\k(A_i \otimes_A V_{g,\alpha})$, then there is some unit $G \in B^* \cong \prod_{\alpha \in E} GL_{m_\alpha}(\k)$ such that for all $a \in A$, $f(a) = G^{-1} g(a) G$.
\end{corollary}

\begin{proof}
    Let $\pi_\alpha: B \to M_{m_\alpha}(\k)$ denote the projection map and $f_\alpha := \pi_\alpha \circ f$, $g_\alpha := \pi_\alpha \circ g$. By Theorem \ref{skolem_noether_quantum}, for each $\alpha \in E$ there is some $G_\alpha \in GL_{m_\alpha}(\k)$ such that $f_\alpha(a) = G^{-1}_\alpha g_\alpha(a) G_\alpha$ for all $a \in A$. By the universal property of the product, letting $G := (G_\alpha)_{\alpha \in E} \in B$, we have $f(a) = G^{-1} g(a) G$ for all $a \in A$.
\end{proof}

Our next goal is to classify QQs up to a good notion of equivalence. Classification up to isomorphism is rather tricky, so we have settled on a weaker notion of equivalence which we'll term ``weak isomorphism". Before we introduce this definition, we need some notation. In general, an homomorphism of finite dimensional $C^*$-algebras takes the form
$$\mu: \prod_{i \in I} M_{n_i}(\C) \to \prod_{j \in J} M_{m_j}(\C)$$
where the $n_i, m_j$ are positive integers and $I,J$ are finite indexing sets. If $\mu$ is an isomorphism, it will induce a bijection $\mu_0: I \to J$ such that $n_i = m_{\mu_0(i)}$. Indeed, a bijection $\sigma: I \to J$ with $n_i = m_{\sigma(i)}$ determines an isomorphism of the above form (although not uniquely).

For a QQ $\scr{B} = (B_0,B_1, \r_B, \s_B)$, denote the indexing set of $B_i$ by $I_i({B}), i = 0, 1.$ To simplify the notation we will often denote the indexing sets as $I_0$ and $I_1$ where no confusion arises.

\begin{definition}
    Let $\scr{B} = (B_0, B_1, \r,\s)$ be a QQ, with $B_0, B_1$ having decompositions;
    $$B_0 = \prod_{v \in I_0} M_{n_v}(\C), \text{ and} $$
    $$B_1 = \prod_{\alpha \in I_1} M_{m_\alpha}(\C).$$
    Let $\frk{t}$ denote either $\r$ or $\s$ and $\frk{t}_\alpha$ denote $\pi_\alpha \circ \frk{t}$ where $\pi_\alpha: B_1 \to M_{m_\alpha}(\C)$ is the projection. By Remark \ref{rem_4_0_1}, we have that for all $v \in I_0, \alpha \in I_1$, $\dim_\C(M_{n_v}(\C) \otimes_{B_0} V_{\frk{t}_\alpha}) $ is an integer multiple of $n_v$, where $V_{\frk{t}_\alpha}$ denotes $\C^{m_\alpha}$ with $B_0$-action induced via $\frk{t}_\alpha$. Call the nonnegative integer
    $$\order(\frk{t},v,\alpha) := \frac{\dim_\C(M_{n_v}(\C) \otimes_{B_0} V_{\frk{t}_\alpha})}{n_v}$$
    the \textit{order of $v$ in $\alpha$ under $\frk{t}$}. By Remark \ref{rem_4_0_1}, we have $\sum_{v \in I_0} \order(\frk{t}, v, \alpha) n_v = m_\alpha$.
\end{definition}

\begin{definition}
    Suppose that for QQ $\scr{B}, \scr{C}$, there are isomorphisms $F_0: B_0 \to C_0$, $F_1: B_1 \to C_1$ which induce bijections $f_0: I_0(\scr{B}) \to I_0(\scr{C})$, $f_1: I_1(\scr{B}) \to I_1(\scr{C})$. We say $(F_0,F_1)$ is a \emph{weak isomorphism} if all $v \in I_0(\scr{B}), \alpha \in I_1(\scr{B})$:
    \begin{enumerate}
        \item $ \order(\r_B, v, \alpha) = \order(\r_C, f_0(v), f_1(\alpha))$, and
        \item $\order(\s_B, v, \alpha) = \order(\s_C, f_0(v), f_1(\alpha))$.
    \end{enumerate}
    We say $\scr{B}$ and $\scr{C}$ are weakly isomorphic if such a weak isomorphism between them exists.
\end{definition}

It is clear that isomorphic QQs are weakly isomorphic. A natural question is then to ask under what cirumstances the converse holds. The next proposition gives some insight into this. If $R$ is a ring, we use the notation $R^*$ to denote the units of $R$.

\begin{proposition}
    Let $\scr{B}, \scr{C}$ be QQs with a weak isomorphism $(F_0, F_1)$ between them. Then there exists $G_r,G_s \in C_1^*$ such that for all $b \in B_0$, $G_r (F_1 \circ \r_B)(b) G_r^{-1} = (\r_C \circ F_0)(b)$ and $G_s (F_1 \circ \s_B)(b) G_s^{-1} = (\s_C \circ F_0)(b)$. In particular, if $C_1$ is commutative, then $(F_0,F_1)$ constitute an isomorphism.
\end{proposition}

\begin{proof}
    Follows from Corollary \ref{skolem_noether_quantum_2}.
\end{proof}

Hence, we see that the inability of weak isomorphism to be an isomorphism is in a sense measured by the inability to choose $G_r,G_s$ such that they are equal up to multiplication by the center of $C_1^*$.

\begin{example}
    Let $B_1 = M_3(\C)$, $B_0 = \C^2$. Let $\s = \r = \r': B_0 \to B_1: (\lambda_1, \lambda_2) \mapsto \diag(\lambda_1, \lambda_1, \lambda_2)$ where $\diag(a,b,c)$ denotes the $3 \times 3$-diagonal matrix with entries $a,b,c$. Let $\s' = B_0 \to B_1: (\lambda_1, \lambda_2) \mapsto \diag(\lambda_1, \lambda_2, \lambda_1)$. Then the QQs $(B_0,B_1, \r,\s)$ and $(B_0,B_1,\r',\s')$ are weakly isomorphic, but not isomorphic.
\end{example}

\begin{theorem} \label{thm4_6}
     Suppose we have two finite dimensional $C^*$-algebras $B_0, B_1$ with $B_0 = \prod_{v \in I_0} M_{n_v}(\C)$ and $B_1 = \prod_{\alpha \in I_1} M_{m_\alpha}(\C)$ and a collection of nonnegative integers $R_{v,\alpha}, S_{v, \alpha}$ such that $\sum_{v \in I_0} n_v R_{v,\alpha} = \sum_{v \in I_0} n_v S_{v,\alpha} = m_\alpha$ for all $\alpha \in I_1$. Then there exist $*$- homomorphisms $\r,\s: B_0 \to B_1$, such that $(B_0,B_1, \r,\s)$ is a QQ with $\order(\r, v, \alpha) = R_{v, \alpha}$ and $\order(\s, v, \alpha) = S_{v, \alpha}$.
\end{theorem}

\begin{proof} 
    The idea is as follows: we define the map $\r: B_0 \to B_1$ by, for each $\alpha \in I_1$ and $v \in I_0$, embedding a copy of $M_{n_v}(\C)$ along the diagonal of $M_{m_\alpha}(\C)$ $R_{v,\alpha}$-times without overlap.

It should be noted that this method of defining $\r$ is only possible since \begin{align*}\sum_{v \in I_0} n_v R_{v,\alpha} = m_\alpha.\end{align*}We define $\s$ similarly, using the $S_{v,\alpha}$ instead of the $R_{v,\alpha}$. Then $\scr{B} = (B_0,B_1,\r,\s)$ is the QQ we desire, and checking that $\order(\r,v,\alpha) = R_{v,\alpha}$ and $\order(\s,v,\alpha) = S_{v,\alpha}$ is routine.
\end{proof}

\begin{remark} \label{rem4_2_1}
    It is clear that the QQ constructed above is unique up to weak isomorphism, but not isomorphism.
\end{remark}

We finish this section by introducing a powerful visual tool for understanding QQ, which we've termed quantum quiver diagrams, which utilizes the ideas behind the above theorem and remark. Given a QQ $\scr{B}$ with $B_0 = \prod_{v \in I_0} M_{n_v}(\C)$ and $B_1 = \prod_{\alpha \in I_1} M_{m_\alpha}(\C)$, we write each of the integers $n_v, v \in I_0$ with a circular border (which we will call the \textit{$v$-node}), and each of the integers $m_\alpha, \alpha \in I_1$ with no border (which we will call the \textit{$\alpha$-node}). Then for each $(v,\alpha) \in I_0 \times I_1$ we draw $\order(\s,v,\alpha)$-lines from the $\alpha$-node to the $v$-node, and $\order(\r,v,\alpha)$-arrows from the $\alpha$-node to the $v$-node.

We illustrate some examples. Throughout the rest of this section, we use $M_n$ as shorthand to denote $M_n(\C)$. Given $B_0 = M_1 \oplus M_3 \oplus M_2 \oplus M_2 \oplus M_4$ and $B_1 = M_4 \oplus M_6$, we wish to construct a QQ on these $C^*$-algebras with orders given by the following tables (a missing entry indicates $0$).

\[ \begin{tabular}{c||c c}
    $\order(\s,v_i, \alpha_j)$ & $\alpha_1$ ($m_{\alpha_1} = 4$) & $\alpha_2$ ($m_{\alpha_2} = 6$) \\
    \hline \hline
     $v_1$ ($n_{v_1} = 1$) & 1 \\
     $v_2$ ($n_{v_2} = 3$) & 1 & 2 \\
     $v_3$ ($n_{v_3} = 2$) & \\
     $v_4$ ($n_{v_4} = 2$) & \\
     $v_5$ ($n_{v_5} = 4$) & \\
\end{tabular} \]

\[ \begin{tabular}{c||c c}
    $\order(\r,v_i, \alpha_j)$ & $\alpha_1$ ($m_{\alpha_1} = 4$) & $\alpha_2$ ($m_{\alpha_2} = 6$) \\
    \hline \hline
     $v_1$ ($n_{v_1} = 1$) & \\
     $v_2$ ($n_{v_2} = 3$) & \\
     $v_3$ ($n_{v_3} = 2$) & 1 \\
     $v_4$ ($n_{v_4} = 2$) & 1 & 1\\
     $v_5$ ($n_{v_5} = 4$) & & 1\\
\end{tabular} \]

Then the associated quantum quiver diagram is as below,
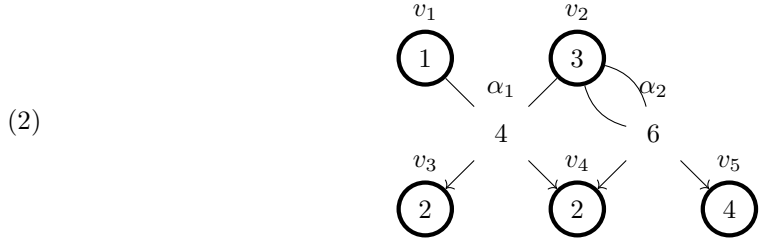
\begin{equation}\label{qqdiagram2}
\begin{aligned} 
\begin{tikzpicture}[
vertex/.style={circle, draw=black, ultra thick, minimum size=7mm},
edge/.style={rectangle, minimum size=7mm},
]
%Nodes
\node[edge] [label = {$\alpha_1$}] (e0) at (0,0) {4};
\node[edge] [label = {$\alpha_2$}] (e1) at (2,0) {6};
\node[vertex] [label = {$v_1$}] (v0) at (-1,1) {1};
\node[vertex] [label = {$v_2$}] (v1) at (1,1) {3};
\node[vertex] [label = {$v_3$}] (v3) at (-1,-1) {2};
\node[vertex] [label = {$v_4$}] (v4) at (1,-1) {2};
\node[vertex] [label = {$v_5$}] (v5) at (3,-1) {4};
%Lines
\draw[-] (v0)--(e0);
\draw[-] (v1)--(e0);
\draw[-] (v1) to [bend left] (e1);
\draw[-] (v1) to [bend right] (e1);
\draw[->] (e0)--(v3);
\draw[->] (e0)--(v4);
\draw[->] (e1)--(v4);
\draw[->] (e1)--(v5);
\end{tikzpicture}
\end{aligned}
\end{equation}

For another example, consider the case where $B_1 = M_1 \oplus M_1 \oplus M_1$ is commutative, where $B_0 = M_2 \oplus M_1 \oplus M_1 \oplus M_1$, and where the orders are given by the following table.

\[ \begin{tabular}{c||c c c}
    $\order(\s,v_i, \alpha_j)$ & $\alpha_1$ ($m_{\alpha_1} = 1$) & $\alpha_2$ ($m_{\alpha_2} = 1$) & $\alpha_3$ ($m_{\alpha_3} = 1$) \\
    \hline \hline
     $v_1$ ($n_{v_1} = 2$) & \\
     $v_2$ ($n_{v_2} = 1$) & 1 & 1 & 1 \\
     $v_3$ ($n_{v_3} = 1$) & \\
     $v_4$ ($n_{v_4} = 1$) & \\
\end{tabular} \]

\[ \begin{tabular}{c||c c c}
    $\order(\r,v_i, \alpha_j)$ & $\alpha_1$ ($m_{\alpha_1} = 1$) & $\alpha_2$ ($m_{\alpha_2} = 1$) & $\alpha_3$ ($m_{\alpha_3} = 1$) \\
    \hline \hline
     $v_1$ ($n_{v_1} = 2$) & \\
     $v_2$ ($n_{v_2} = 1$) & \\
     $v_3$ ($n_{v_3} = 1$) & 1 \\
     $v_4$ ($n_{v_4} = 1$) & & 1 & 1\\
\end{tabular} \]

Then the associated quantum quiver is then given by the following.
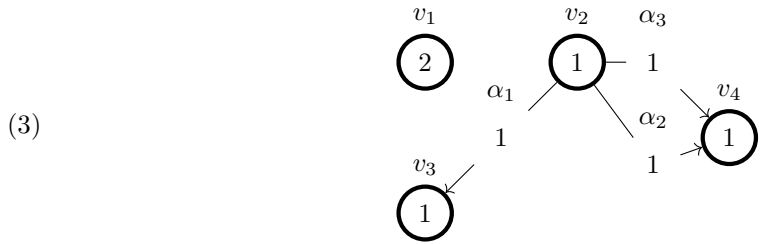
\begin{equation}
\begin{aligned}
\begin{tikzpicture}[
vertex/.style={circle, draw=black, ultra thick, minimum size=7mm},
edge/.style={rectangle, minimum size=7mm},
]
%Nodes
\node[edge] [label = {$\alpha_1$}] (e0) at (-1,1) {1};
\node[edge] [label = {$\alpha_2$}, below] (e1) at (1,1) {1};
\node[edge] [label = {$\alpha_3$}] (e2) at (1,2) {1};
\node[vertex] [label = {$v_1$}] (v0) at (-2,2) {2};
\node[vertex] [label = {$v_2$}] (v1) at (0,2) {1};
\node[vertex] [label = {$v_3$}] (v2) at (-2,0) {1};
\node[vertex] [label = {$v_4$}] (v3) at (2,1) {1};
%Lines
\draw[-] (v1)--(e0);
\draw[-] (v1)--(e1);
\draw[-] (v1)--(e2);
\draw[->] (e0)--(v2);
\draw[->] (e1)--(v3);
\draw[->] (e2)--(v3);
\end{tikzpicture}
\end{aligned}
\end{equation}

Note that if we remove the copy of $M_2$ from $B_0$, making both $B_0,B_1$ commutative, we recover the following directed graph diagram.
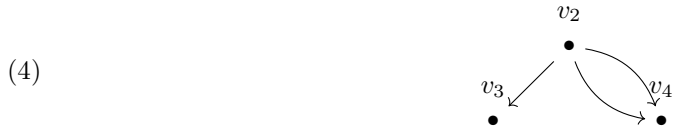
\begin{equation}
\begin{aligned}
\begin{tikzpicture}
%Nodes
\node (v1) [label = {$v_2$}] at (0,1) {$\bullet$};
\node (v2) [label = {$v_3$}] at (-1,0) {$\bullet$};
\node (v3) [label = {$v_4$}, right] at (1,0) {$\bullet$};
%Lines
\draw[->] (v1)--(v2);
\draw[->] (v1) to [bend left] (v3);
\draw[->] (v1) to [bend right] (v3);
\end{tikzpicture}
\end{aligned}
\end{equation}
As one would expect, this happens generally. If $B_0,B_1$ are commutative, then the quantum quiver diagram coincides with the usual directed graph diagram.

\section{Applications to Leavitt path algebras}

 In this section we apply QQ in the context of Leavitt path algebras. As the name suggests, this involves constructing a suitable notion of paths in the noncommutative context, which we do using properties of matrix decompositions of $C^*$-algebras.

\begin{definition}
    Call a (non-unital) homomorphism of matrix algebras $f: M_n(\K) \to M_m(\K)$ \emph{regular} if it embeds one or more copies of $M_n(\K)$ along the diagonal of $M_m(\K)$. We trivially say that a zero map is regular. Call a map of finite dimensional $C^*$-algebras $f: \bigoplus_i M_{n_i}(\C) \to \bigoplus_j M_{n_j}(\C)$ \emph{regular} if for all $i,j$ the induced map $f_{i,j} : M_{n_i}(\C) \to M_{m_j}(\C)$ is regular. Call a QQ $B = (B_0,B_1, \r,\s)$ \emph{regular} if both $\r,\s$ are regular.
\end{definition}

Note that a homomorphism $f: M_n(\K) \to M_m(\K)$ is regular if and only if it can be factored as $\iota \circ \Delta$ where $\Delta: M_n(\K) \to M_n(\K)^N$ is a diagonal embedding for some $N$, and $\iota: M_n(\K)^N \to M_m(\K)$ is an inclusion map (implying $nN \leq m$).

We tacitly assume for the rest of this section that a QQ $(B_0,B_1,\r,\s)$, which will just be denoted by $\scr{B}$, has $B_0, B_1$ having the decompositions:

\begin{enumerate}
    \item $B_0 = \prod_{v \in I_0} M_{n_v}(\C)$, and
    \item $B_1 = \prod_{\alpha \in I_1} M_{m_\alpha}(\C)$,
\end{enumerate}

where $I_0,I_1$ are finite indexing sets. Hence $\scr{B}$ being regular implies $\r: B_0 \to B_1$ can be written in the form $e^v_{ij} \mapsto \sum_{(\alpha, i', j') \in R(v,i,j)} e^\alpha_{i'j'}$ where $R(v,i,j) \subset \scr{E}_1 := \{(\alpha, i',j') : \alpha \in I_1, 1 \leq i',j' \leq m_\alpha\}$. In this way we see $R$ as a function on $\scr{E}_0 := \{(v,i,j): v \in I_0, 1 \leq i,j \leq n_v\}$ which takes values in subsets of $\scr{E}_1$. We may similarly define $S: \scr{E}_0 \to 2^{\scr{E}_1}$ such that $\s(e^v_{ij}) = \sum_{(\alpha,i',j') \in S(v,i,j)} e^\alpha_{i'j'}$.

\begin{lemma} \label{lem_4_3}
    Let $\scr{B}$ be a regular QQ. Then for all generators $e^{\alpha}_{i'j'} \in B_1$, there is at most one $e^v_{ij} \in B_0$ such that $(\alpha,i',j') \in R(v,i,j)$. Moreover, if $i' = j'$, there is precisely one $e^v_{ij} \in B_0$ where $i = j$ such that $(\alpha, i',j') \in R(v,i,j)$.
\end{lemma}

Before we state the proof, it is important to note that this lemma still applies if we replace $R$ with $S$, and $\r$ with $\s$.

\begin{proof}
    The second statement follows from $\r$ preserving units. To see the first statement, suppose we have $(\alpha, i',j') \in \scr{E}_1$, $(v_1,i_1,j_1), (v_2,i_2,j_2) \in \scr{E}_0$ such that $(\alpha, i',j') \in R(v_1,i_1,j_1) \cap R(v_2,i_2,j_2)$. By definition, we may write
    \begin{align*}
        \r(e^{v_1}_{i_1 j_1}) & = e^\alpha_{i'j'} + \sum_{(\beta, k,l) \in R(v_1,i_1,j_1) - \{(\alpha, i',j')\}} e^\beta_{kl}, \\
        \r(e^{v_2}_{i_2 j_2}) & = e^\alpha_{i'j'} + \sum_{(\beta, k,l) \in R(v_2,i_2,j_2) - \{(\alpha, i',j')\}} e^\beta_{kl},
    \end{align*}
    where $\gamma_1 := \sum_{(\beta, k,l) \in R(v_1,i_1,j_1) - \{(\alpha, i',j')\}} e^\beta_{kl}$ is orthogonal to $e^\alpha_{i'j'}$ and 
    
    $\gamma_2 := \sum_{(\beta, k,l) \in R(v_2,i_2,j_2) - \{(\alpha, i',j')\}} e^\beta_{kl}$ is orthogonal to $e^\alpha_{i'j'}$. Since $\r$ is regular, it preserves taking transposes. Combining this with $\r$ preserving multiplication yields
    \begin{align*}
        \r(e^{v_1}_{i_1 i_1}) & = e^\alpha_{i'i'} + \sum_{(\beta, k,l) \in R(v_1,i_1,i_1) - \{(\alpha, i',i')\}} e^\beta_{kl}, \\
        \r(e^{v_2}_{i_2 i_2}) & = e^\alpha_{i'i'} + \sum_{(\beta, k,l) \in R(v_2,i_2,i_2) - \{(\alpha, i',i')\}} e^\beta_{kl}
    \end{align*}
    which contradicts the second statement.
\end{proof}

This gives us a useful way to characterize the adjoint map.

\begin{theorem}
    For a regular QQ $\scr{B}$, the adjoint range map $\r^*: B_1 \to B_0$ takes the form
    $$e^\alpha_{ij} \mapsto \begin{cases}
        e^v_{i',j'} & \text{for some $v \in I_0, 1 \leq i',j' \leq n_v$ such that $(\alpha,i,j) \in R(v,i',j')$,} \\
        0 & \text{if no such $(v,i',j')$ exists.}
    \end{cases}$$
\end{theorem}

\begin{proof}
    The assignment is well-defined by Lemma \ref{lem_4_3}. The lemma then follows from a more general result for Hilbert spaces. Let $\scr{H}_1, \scr{H}_2$ be finite dimensional Hilbert spaces, with inner products $\inner{-}{-}_1, \inner{-}{-}_2$ respectively, and orthonormal bases $\{v_1,...,v_n\}, \{w_1,...,w_m\}$ respectively. Then the adjoint of a linear map $f: \scr{H}_1 \to \scr{H}_2$ (seen as an $m \times n$ matrix) is given by it's conjugate transpose. Since the $e^{\alpha}_{ij}, e^v_{ij}$ are orthonormal bases (recall the inner-product we use is $\inner{A}{B} = \operatorname{trace}(A^*B)$) for $B_1,B_0$ respectively, the result follows.
\end{proof}

We remind the reader of the definition of the Leavitt path algebra.

\begin{definition} Fix a field $\K$. The \textit{Leavitt path algebra} of a directed graph $E = (E_0,E_1,r,s)$ over $\K$, denoted $L_\K(E)$ is the associative $\K$-algebra on symbols $\{p_v: v \in E_0\} \cup E_1 \cup E_1^*$ (where $E_1^*$ is the set of symbols $\{e^*: e \in E_1\}$) subject to the following relations;
    \begin{enumerate}
        \item[L1.] $p_v p_w = \delta_{v,w} p_v$ for $v \in E_0$, where $\delta_{v,w}$ is the Kronecker delta given by 1 if $v=w$ and 0 otherwise,
        \item[L2.] $p_{s(e)} e = e = e p_{r(e)}$ for $e \in E_1$,
        \item[L3.] $p_{r(e)} e^* = e^* = e^* p_{s(e)}$ for $e \in E_1$,
        \item[L4.] $e^* f = \delta_{e,f} p_{r(f)}$ for $e,f \in E_1$, and
        \item[L5.] $p_v = \sum_{e \in s^{-1}(v)} e e^*$ for  every regular vertex $v \in E_0$.
    \end{enumerate}
\end{definition}

The LPA comes equipped with an involution defined by $e \mapsto e^*$, $e^* \mapsto e$ and $p_v \mapsto p_v$, and a $\Z$-grading defined by $|e| = -|e^*| = 1$ and $|p_v| = 0$. It can be readily observed that $L_\K(E)$ is unital with $1 = \sum_{v \in E_0} p_v$ (note that this is only true since $E_0$ is assumed to be finite).

\begin{example}
    Let $E$ denote the quiver with a single vertex $v$ and single edge $e$, so $r(e) = s(e) = v$. It is easy to observe then that $L_\K(E)$ has unit $1=p_v$ and is hence the free algebra on $\{e, e^*\}$ subject to the relation $ee^* = e^*e = 1$. Thus we have an isomorphism $L_\K(E) \to \K[t,t^{-1}]$ via $e \mapsto t, e^* \mapsto t^{-1}$. Moreover, this isomorphism preserves the grading.
\end{example}

\begin{example}
    More generally, if $E$ denotes the quiver of one vertex $v$ and $n$ edges $\{e_1,...,e_n\}$ each with source and range $v$, then for a field $\K$, $L_\K(E)$ is the \textit{Leavitt algebra} $L_\K(n)$ which by \cite{Leavitt1962} has \textit{module type} $(1,n)$; $n$ is the least integer $r > 1$ for which $L_\K(E) \cong L_\K(E)^r$ as left $L_\K(E)$-modules. One can see this as an algebraic formulation of the Cuntz algebra $\cal{O}_n$.
\end{example}

In the same way as one may derive Leavitt path algebras as algebraic formulation of graph $C^*$-algebras, we can do the same here by placing generators in Definition \ref{defn_free_quantum_B_family} and seeing what relations we get. We note that this process works very elegantly in the case of Leavitt path algebras because of how much orthogonality one has. The same is not true here: the algebra we'll create by looking at the relations imposed by Definition \ref{defn_free_quantum_B_family} is incredibly complex with relations often involving several generators. The purpose of us doing this is to study its algebraic $K_0$ value.

With this background, we now introduce analogues of classical range and source maps for quantum quivers, from which we are able to define Leavitt Path algebras over quantum quivers, and correspondingly define graph theoretic notions such as paths. We first fix the following notation:

\begin{enumerate}
    \item $r(\sigma^\alpha_{ij}) = \begin{cases}
        \rho^v_{i'j'} & \text{where $(\alpha,i,j) \in R(v,i',j')$} \\
        0 & \text{if no such $(v,i',j')$ exists},
    \end{cases}$
    \item $s(\sigma^\alpha_{ij}) = \begin{cases}
        \rho^v_{i'j'} & \text{where $(\alpha,i,j) \in S(v,i',j')$} \\
        0 & \text{if no such $(v,i',j')$ exists},
    \end{cases}$
    \item $r(\bar{\sigma}^\alpha_{ij}) = s(\sigma^\alpha_{ij})$, and
    \item $s(\bar{\sigma}^\alpha_{ij}) = r(\sigma^\alpha_{ij})$.
\end{enumerate}

\begin{definition} \label{defn_3_7_LkB}
    Let $\K$ be a field and $\scr{B} = (B_0,B_1,\r,\s)$ a regular QQ. Additionally, assume we have decompositions,  
    \begin{align*}
        B_0 = \prod_{v \in I_0} M_{n_v}(\C), \quad  B_1 = \prod_{\alpha \in I_1} M_{m_\alpha}(\C).
    \end{align*} 
    For a fixed integer $n$, denote $I_{0,n} = \{v \in I_0: n_v = n\}$. 
    %Define the \textit{Leavitt path algebra} associated with $\scr{B}$, $L_\K(\scr{B})$, to be 
    Let $L_\K(\scr{B})$ denote the $\K$-algebra on generators $\{\rho^v_{ij}: v \in I_0, 1 \leq i,j \leq n_v\}$ and $\{\bar{\sigma}^\alpha_{ij}, \sigma^\alpha_{ij}: \alpha \in I_1, 1 \leq i,j \leq m_\alpha\}$, subject to the relations:
    \begin{enumerate}
        \item for all integers $n$ and $v,w \in I_{0,n}$,
        $\sum_{j = 1}^{n_v} \rho^v_{ij} \rho^w_{jk} = \delta_{v,w} \rho^v_{ik}$,
        
        \item for all integers $n$ such that $I_{0,n}$ is nonempty, $\sum_{v \in I_{0,n}} \rho^v_{ij} = \delta_{ij}$,
        
        \item for all $\alpha \in I_1$,
        $\sum_{j=1}^{m_\alpha} \bar{\sigma}^\alpha_{ij} \sigma^\alpha_{jk} = r(\sigma^\alpha_{ik}),$
        
        \item for all $v \in I_0$,
        $\rho^v_{ik} = \sum_{(\alpha,i',k') \in S(v,i,k)} \sum_{j=1}^{m_\alpha} \sigma^\alpha_{i',j} \bar{\sigma}^\alpha_{j,k'},$ where $S(v, i, k) $ is not empty,
        
        \item for all $\alpha \in I_1$, 
        $\sigma^\alpha_{ik} = \sum_{j = 1}^{m_\alpha} \sigma^\alpha_{ij} r(\sigma^\alpha_{jk}) = \sum_{j=1}^{m_\alpha} s(\sigma^\alpha_{ij}) \sigma^\alpha_{jk},$  and
        $\bar{\sigma}^\alpha_{ik} = \sum_{j = 1}^{m_\alpha} \bar{\sigma}^\alpha_{ij} r(\bar{\sigma}^\alpha_{jk}) = \sum_{j=1}^{m_\alpha} s(\bar{\sigma}^\alpha_{ij}) \bar{\sigma}^\alpha_{jk}$.
    \end{enumerate}
    We call $L_\K(\scr{B})$ the Leavitt path algebra of $\scr{B}$ over $\k$.
\end{definition}

Note that the symbols $r, s$ in the above equations are an analogue of the range and source maps of the generators $\{\bar{\sigma}^\alpha_{ij}, \sigma^\alpha_{ij}: \alpha \in I_1, 1 \leq i,j \leq m_\alpha\}$.

\begin{remark}
    In the case where the $C^*$-algebras in $\scr{B}$ are commutative (i.e., $\scr{B}$ is a directed graph), the algebra $L_\K(\scr{B})$ coincides with the Leavitt path algebra of the associated directed graph.
\end{remark}

We'll discuss an alternative way of interpreting the generators and relations present in $L_\K(\scr{B})$. We'll do this with respect to the following snippet of the quantum quiver diagram (\ref{qqdiagram2}):

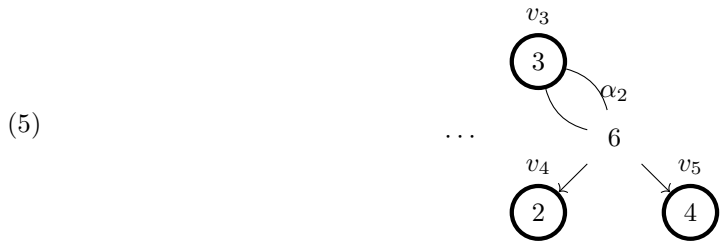
\begin{equation}
\begin{aligned} 
\begin{tikzpicture}[
vertex/.style={circle, draw=black, ultra thick, minimum size=7mm},
edge/.style={rectangle, minimum size=7mm},
]
%Nodes
\node[edge] (e0) at (0,0) {$\cdots$};
\node[edge] [label = {$\alpha_2$}] (e1) at (2,0) {6};
\node[vertex] [label = {$v_3$}] (v1) at (1, 1) {3};
\node[vertex] [label = {$v_4$}] (v4) at (1,-1) {2};
\node[vertex] [label = {$v_5$}] (v5) at (3,-1) {4};
%Lines
\draw[-] (v1) to [bend left] (e1);
\draw[-] (v1) to [bend right] (e1);
\draw[->] (e1)--(v4);
\draw[->] (e1)--(v5);
\end{tikzpicture}
\end{aligned}
\end{equation}

We associate to each of the vertices matrices according to their size:
\begin{enumerate}
    \item $v_3 \rightsquigarrow V_3 := \begin{pmatrix}
        \rho^3_{11} & \rho^3_{12} & \rho^3_{13} \\
        \rho^3_{21} & \rho^3_{22} & \rho^3_{23} \\
        \rho^3_{31} & \rho^3_{32} & \rho^3_{33}
    \end{pmatrix}$,
    \item $v_4 \rightsquigarrow V_4 := \begin{pmatrix}
        \rho^4_{11} & \rho^4_{12} \\
        \rho^4_{21} & \rho^4_{22} 
    \end{pmatrix}$, and
    \item $v_5 \rightsquigarrow V_5 := \begin{pmatrix}
        \rho^5_{11} & \rho^5_{12} & \rho^5_{13} & \rho^5_{14} \\
        \rho^5_{21} & \rho^5_{22} & \rho^5_{23} & \rho^5_{24} \\
        \rho^5_{31} & \rho^5_{32} & \rho^5_{33} & \rho^5_{14} \\
        \rho^5_{41} & \rho^5_{42} & \rho^5_{43} & \rho^5_{44}
    \end{pmatrix}$,
\end{enumerate}
and associate to each edge a square matrix and ``conjugate transpose" square matrix according to its size
$$\alpha_2 \rightsquigarrow E_2 := \begin{pmatrix}
    \sigma^2_{11} & \sigma^2_{12} & \sigma^2_{13} & \sigma^2_{14} & \sigma^2_{15} & \sigma^2_{16} \\
    \sigma^2_{21} & \sigma^2_{22} & \sigma^2_{23} & \sigma^2_{24} & \sigma^2_{25} & \sigma^2_{26} \\
    \sigma^2_{31} & \sigma^2_{32} & \sigma^2_{33} & \sigma^2_{34} & \sigma^2_{35} & \sigma^2_{36} \\
    \sigma^2_{41} & \sigma^2_{42} & \sigma^2_{43} & \sigma^2_{44} & \sigma^2_{45} & \sigma^2_{46} \\
    \sigma^2_{51} & \sigma^2_{52} & \sigma^2_{53} & \sigma^2_{54} & \sigma^2_{55} & \sigma^2_{56} \\
    \sigma^2_{61} & \sigma^2_{62} & \sigma^2_{63} & \sigma^2_{64} & \sigma^2_{65} & \sigma^2_{66}
\end{pmatrix}, E_2^* := \begin{pmatrix}
    \bar{\sigma}^2_{11} & \bar{\sigma}^2_{12} & \bar{\sigma}^2_{13} & \bar{\sigma}^2_{14} & \bar{\sigma}^2_{15} & \bar{\sigma}^2_{16} \\
    \bar{\sigma}^2_{21} & \bar{\sigma}^2_{22} & \bar{\sigma}^2_{23} & \bar{\sigma}^2_{24} & \bar{\sigma}^2_{25} & \bar{\sigma}^2_{26} \\
    \bar{\sigma}^2_{31} & \bar{\sigma}^2_{32} & \bar{\sigma}^2_{33} & \bar{\sigma}^2_{34} & \bar{\sigma}^2_{35} & \bar{\sigma}^2_{36} \\
    \bar{\sigma}^2_{41} & \bar{\sigma}^2_{42} & \bar{\sigma}^2_{43} & \bar{\sigma}^2_{44} & \bar{\sigma}^2_{45} & \bar{\sigma}^2_{46} \\
    \bar{\sigma}^2_{51} & \bar{\sigma}^2_{52} & \bar{\sigma}^2_{53} & \bar{\sigma}^2_{54} & \bar{\sigma}^2_{55} & \bar{\sigma}^2_{56} \\
    \bar{\sigma}^2_{61} & \bar{\sigma}^2_{62} & \bar{\sigma}^2_{63} & \bar{\sigma}^2_{64} & \bar{\sigma}^2_{65} & \bar{\sigma}^2_{66}
\end{pmatrix}^T.$$

For square matrices $A,B$ of size $n \times n$ and $m \times m$ respectively, by $A \square B$ we mean the $(n+m) \times (n+m)$-matrix $\begin{pmatrix}
    A & 0 \\
    0 & B
\end{pmatrix}$. The relations of $L_\K(\scr{B})$ become the following
\begin{enumerate}
    \item[1.] $V_i^2 = V_i$ for all $i$,
    \item[3.] $E_2^* E_2 = V_4 \square V_5$,
    \item[4.] $V_3 \square V_3 = E_2 E_2^*$,
    \item[5.] $(V_3 \square V_3) E_2 = E_2 = E_2 (V_4 \square V_5)$ and $(V_4 \square V_5) E_2^* = E_2^* = E_2^* (V_3 \square V_3)$.
\end{enumerate}
Rewriting relation (2) requires global information about $\scr{B}$, but letting $V_i$, $i=1,2$ denote matrices $\{\rho^v_{jk}\}_{1 \leq j,k \leq n_{v_i}}$ in a similar manner to how $V_3,V_4,V_5$ were defined, relation (2) becomes
\begin{enumerate}
    \item $V_1 = (1)$ (the $1 \times 1$ identity matrix),
    \item $V_2 + V_4 = \begin{pmatrix}
        1 & 0 \\ 0 & 1
    \end{pmatrix}$,
    \item $V_3 = \begin{pmatrix}
        1 & 0 & 0 \\ 0 & 1 & 0 \\ 0 & 0 & 1
    \end{pmatrix}$, and
    \item $V_5 = \begin{pmatrix}
        1 & 0 & 0 & 0 \\ 0 & 1 & 0 & 0 \\ 0 & 0 & 1 & 0 \\ 0 & 0 & 0 & 1
    \end{pmatrix}$.
\end{enumerate}

\begin{example}
    A Leavitt algebra $L_\K(n,m)$ for integers $1 < n < m$ is a $\k$-algebra of module type $(n,m)$, meaning $m$ is the least integer $r > n$ such that $L_\K(n,m)^n \cong L_\K(n,m)^r$. A subclass of these algebras are contained in our construction. Specifically if $n,k$ are positive integers with $k > 1$ we may form the following QQ diagram
    \begin{equation}
    \begin{aligned}
    \begin{tikzpicture}[
    vertex/.style={circle, draw=black, ultra thick, minimum size=7mm},
    edge/.style={rectangle, minimum size=7mm},
    ]
    %Nodes
    \node[vertex] [label = {$v_n$}] (vn) at (-2,0) {$n$};
    \node[vertex] [label = {$v_{kn}$}] (vkn) at (2,0) {$kn$};
    \node[edge] (e) at (0,0) {$kn$};
    \node at (-1,0) {$\vdots$};
    %Lines
    \draw[transform canvas={yshift=2.5mm},-] (vn)--(e);
    \draw[transform canvas={yshift=-2.5mm},-] (vn)--(e);
    \draw[->] (e)--(vkn);
    \end{tikzpicture}
    \end{aligned}
    \end{equation}
    where $\vdots$ denotes $k$ parallel lines. Then by Theorem \ref{thm_4_2_projective_module_computation} below the corresponding Leavitt path algebra has module type $(kn,k^2 n)$.
\end{example}

\section{Projective modules over Leavitt path algebras of quantum quivers} 

Let $\scr{B} = (B_0,B_1,\r,\s)$ be a quantum graph with matrix blocks in $B_i$ indexed by a finite set $I_i$ for $i=0,1$. In this section we give an explicit computation of the monoid of projective modules over the Leavitt path algebra we've defined for $\scr{B}$. We first put an equivalence relation on $I_0$. If two matrix blocks $C,D$ in $B_0$ have the images $C_0 + C_1 + \dots + C_n$ and $D_0 + D_1 + \dots + D_m$ under $\s$ (the $C_i$'s and $D_i$'s are matrices in matrix blocks in $B_1$), where (without loss of generality) $C_0, D_0$ exist in the same matrix block in $B_1$, we call $C,D$ \emph{source equivalent}. Call $v,w \in I_0$ \emph{source-equivalent} if their corresponding matrix blocks in $B_0$ are source equivalent, and let $\sim_{\s}$ denote the equivalence relation generated by source equivalence.

Moreover, for each matrix block in $B_0$ with index $v$ and matrix block in $B_1$ with index $\alpha$, recall the integers $\order(\s,v,\alpha)$ which essentially tells us how many times $\s$ ``embeds" a copy of block $v$ in the block $\alpha$ (strictly speaking, we should say \emph{block indexed by $v$}, but the sentence reads better this way). For the following proof, we shall refer to the collection of these integers as the \emph{out-degrees} of $v$.

Before, we calculate the monoid of finitely generated projective modules we need to make a note of the generators we have. For each $v \in I_0 $, we have a matrix block in $B_0$ with size $n_v$. Since in the Leavitt path algebra this gives us an $n_v \times n_v$ idempotent matrix we have a generator $P_v$ in the monoid which is a summand of the free module of rank $n_v$. We have a complementary idempotent $I_{n_v} - P_v$ (where $I_{n_v}$ denotes the identity matrix) which, in the same manner as above, corresponds to a projective module $Q_v$ and we have the relation $P_v \oplus Q_v \cong n_v I$ where $I$ is the generator corresponding to the free module of rank one. We call relations of this form \emph{trivial}, since these hold in the monoid of finitely generated projective modules over any ring.

\begin{theorem} \label{thm_4_2_projective_module_computation}
    For each source-equivalent class $[v]$, let $N_{[v]}$ denote the LCM of all nonzero out-degrees of all vertices in $[v]$. Then $L_k(\scr{B})$ has monoid of finitely generated projective modules whose only \emph{nontrivial relation} is
    $$ N_{[v]}\left(\bigoplus_{w \in [v]} P_w\right) \cong N_{[v]} \left(\bigoplus_{e \in E_{[v]}} \bigoplus_{u \in R(e)} P_u\right).$$ 
\end{theorem}

\begin{proof}

Throughout, let $q$ denote $N_{[v]}$ (for shorthand). It suffices to find a matrices $A,B$ (with entries in $L_k(\scr{B})$) such that $AB = \bigoplus_{w \in[v]} [\rho^w]^{\oplus q} $, $BA = \bigoplus_{e \in E_{[v]}} \bigoplus_{w \in R(e)} [\rho^w]^{\oplus q} $, $ABA = A$ and $BAB = B$. We let $A$ denote a matrix with rows indexed by the set $[v]_E^q$ (i.e., $q$ copies of each vertex $w$ in the equivalence class $[v]_E$). In the row indexed by $w$, we let $A$ have nonzero entries $\sigma^\alpha_{ij}$ such that the sum of the row is $\sum_{(\alpha,i) \in S(w)} \sum_{j=1}^{m_\alpha} \sigma_{ij}^\alpha$. We align these nonzero entries such that columns sum to $\sum_{i=1}^{m_\alpha} \sigma_{ij}^\alpha$ for some fixed $\alpha, j$ (i.e., if we leave a column fixed, $j, \alpha$ don't vary, and there are precisely $m_\alpha$ nonzero entries in the aforementioned column, given by $\sigma^\alpha_{ij}$ for all $1 \leq i \leq m_\alpha$). We now let $B$ denote the involution tranpose of $A$, and it is easily checked that $A,B$ satisfy the required conditions.

To see that there are no further nontrivial relations, we note that writing the relations of $L_\k(\scr{B})$ in matrix form, similar to the section after Definition \ref{defn_3_7_LkB}, and considering the algebra generated by the matrices with formal addition and multiplication as determined by the matrix relations, no new projective modules or relations are added, and so with the constructions and ideas in \cite[\S 5]{Bergman1974} we conclude there are no further nontrivial relations beyond what was proved above.
\end{proof}

\begin{remark}
    We need the matrices $A$ and $B$ above to have rows and columns indexed by $N_{[v]}$ copies of the involved vertices in order for us to align the nonzero entries in the required way.
\end{remark}

The reader should note in the case where $B_0,B_1$ are commutative (i.e., the case of a directed graph) we recover the graph monoid as the monoid of finitely generated projective modules over the Leavitt path algebra.

In the sequel we illustrate the above theorem with an example. Note that while projective modules can still be computed relatively easily, the computations can become a unwieldy, 

\begin{example}
    We return to the example of Diagram (\ref{qqdiagram2}), which we include here for convenience of the reader. Let $\scr{B}$ be a regular Leavitt quantum graph associated to the following diagram.
    
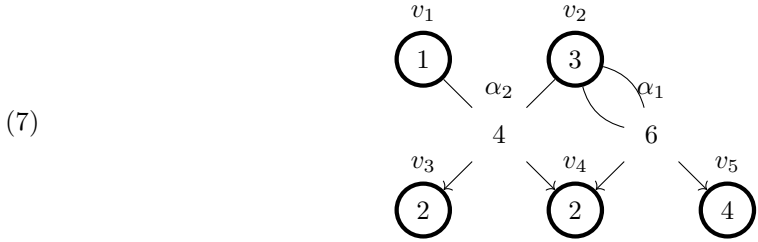
\begin{equation}
    \begin{aligned}
    \begin{tikzpicture}[
    vertex/.style={circle, draw=black, ultra thick, minimum size=7mm},
    edge/.style={rectangle, minimum size=7mm},
    ]
    %Nodes
    \node[edge] [label = {$\alpha_2$}] (e0) at (0,0) {4};
    \node[edge] [label = {$\alpha_1$}] (e1) at (2,0) {6};
    \node[vertex] [label = {$v_1$}] (v0) at (-1,1) {1};
    \node[vertex] [label = {$v_2$}] (v1) at (1,1) {3};
    \node[vertex] [label = {$v_3$}] (v3) at (-1,-1) {2};
    \node[vertex] [label = {$v_4$}] (v4) at (1,-1) {2};
    \node[vertex] [label = {$v_5$}] (v5) at (3,-1) {4};
    %Lines
    \draw[-] (v0)--(e0);
    \draw[-] (v1)--(e0);
    \draw[-] (v1) to [bend left] (e1);
    \draw[-] (v1) to [bend right] (e1);
    \draw[->] (e0)--(v3);
    \draw[->] (e0)--(v4);
    \draw[->] (e1)--(v4);
    \draw[->] (e1)--(v5);
    \end{tikzpicture}
\end{aligned}
    \end{equation}
    We may choose a regular representation of this quiver by defining $\s: B_0 \to B_1$ by
    $$\begin{cases}
        e^{v_1} & \mapsto e^{\alpha_2}_{11} \\
        e^{v_2}_{ij} & \mapsto e^{\alpha_1}_{i+1,j+1} + e^{\alpha_2}_{ij} + e^{\alpha_2}_{i+3,j+3} \\
        e^{v_3}_{ij}, e^{v_4}_{ij}, e^{v_3}_{ij} & \mapsto 0
    \end{cases}$$
    and $\r: B_0 \to B_1$ by
    $$\begin{cases}
        e^{v_1},e^{v_2}_{ij} & \mapsto 0 \\
        e^{v_3}_{ij} & \mapsto e^{\alpha_2}_{ij} \\
        e^{v_4}_{ij} & \mapsto e^{\alpha_1}_{ij} + e^{\alpha_2}_{i+2,j+2} \\
        e^{v_5}_{ij} & \mapsto e^{\alpha_1}_{i+2,j+2}.
    \end{cases}$$
    Fixing a field $\K$, we may trace through the definition of $L_\K(\scr{B})$ to find that $\rho^v_{ij} = \delta_{ij}$ for $v=v_1,v_2,v_5$, while
    $$e_1 := \begin{pmatrix} \rho^{v_3}_{11} & \rho^{v_3}_{12} \\ \rho^{v_3}_{21} & \rho^{v_3}_{22} \end{pmatrix}, e_2 := \begin{pmatrix} \rho^{v_4}_{11} & \rho^{v_4}_{12} \\ \rho^{v_4}_{21} & \rho^{v_4}_{22} \end{pmatrix} \in M_2(L_\K(\scr{B})) $$
    are orthogonal idempotents. Hence, the monoid of projective modules $\cal{P} = \cal{P}(L_\K(\scr{B}))$ is generated by $I$ the free module of rank 1, and two generators $P,Q$ (each corresponding to the images of the matrices above respectively) such that $P + Q = 2I$. The one other relation we have on $\cal{P}$ is given by $8I = 2P + 4Q + 8I = 12I + 2Q$, since we have the following relations between matrices:
    \begin{enumerate}
        \item $C_1 := e_1 \oplus e_2 \oplus e_1 \oplus e_2 \oplus I_4 \oplus e_2 \oplus I_4 \oplus e_2 = AB$,
        \item $C_2 := I_1 \oplus I_3 \oplus I_1 \oplus I_3 = I_8 = BA$,
        \item $C_1 A = A C_2 = A$, and
        \item $C_2 B = B C_1 = B$,
    \end{enumerate}
    where $e_1 \oplus ...$ denotes the usual direct sum of matrices, and $I_n$ denotes the $n \times n$ identity matrix, and $A,B$ are the martices:

 \vspace{2mm}
    
    \scalebox{0.80}{$B=\begin{pmatrix}
        \sigma^{\alpha_2}_{11} & \sigma^{\alpha_2}_{12} & \sigma^{\alpha_2}_{13} & \sigma^{\alpha_2}_{14} & 0 & 0 & 0 & 0 & 0 & 0 & 0 & 0 & 0 & 0 & 0 & 0 & 0 & 0 & 0 & 0 \\
        
        \sigma^{\alpha_2}_{21} & \sigma^{\alpha_2}_{22} & \sigma^{\alpha_2}_{23} & \sigma^{\alpha_2}_{24} & 0 & 0 & 0 & 0 & \sigma^{\alpha_1}_{11} & \sigma^{\alpha_1}_{12} & \sigma^{\alpha_1}_{13} & \sigma^{\alpha_1}_{14} & \sigma^{\alpha_1}_{15} & \sigma^{\alpha_1}_{16} & \sigma^{\alpha_1}_{41} & \sigma^{\alpha_1}_{42} & \sigma^{\alpha_1}_{43} & \sigma^{\alpha_1}_{44} & \sigma^{\alpha_1}_{45} & \sigma^{\alpha_1}_{46} \\
        
        \sigma^{\alpha_2}_{31} & \sigma^{\alpha_2}_{32} & \sigma^{\alpha_2}_{33} & \sigma^{\alpha_2}_{34} & 0 & 0 & 0 & 0 & \sigma^{\alpha_1}_{21} & \sigma^{\alpha_1}_{22} & \sigma^{\alpha_1}_{23} & \sigma^{\alpha_1}_{24} & \sigma^{\alpha_1}_{25} & \sigma^{\alpha_1}_{26} & \sigma^{\alpha_1}_{51} & \sigma^{\alpha_1}_{52} & \sigma^{\alpha_1}_{53} & \sigma^{\alpha_1}_{54} & \sigma^{\alpha_1}_{55} & \sigma^{\alpha_1}_{56} \\
        
        \sigma^{\alpha_2}_{41} & \sigma^{\alpha_2}_{42} & \sigma^{\alpha_2}_{43} & \sigma^{\alpha_2}_{44} & 0 & 0 & 0 & 0 & \sigma^{\alpha_1}_{31} & \sigma^{\alpha_1}_{32} & \sigma^{\alpha_1}_{33} & \sigma^{\alpha_1}_{34} & \sigma^{\alpha_1}_{35} & \sigma^{\alpha_1}_{36} & \sigma^{\alpha_1}_{61} & \sigma^{\alpha_1}_{62} & \sigma^{\alpha_1}_{63} & \sigma^{\alpha_1}_{64} & \sigma^{\alpha_1}_{65} & \sigma^{\alpha_1}_{66} \\
        
        0 & 0 & 0 & 0 & \sigma^{\alpha_2}_{11} & \sigma^{\alpha_2}_{12} & \sigma^{\alpha_2}_{13} & \sigma^{\alpha_2}_{14} & 0 & 0 & 0 & 0 & 0 & 0 & 0 & 0 & 0 & 0 & 0 & 0 \\
        
        0 & 0 & 0 & 0 & \sigma^{\alpha_2}_{21} & \sigma^{\alpha_2}_{22} & \sigma^{\alpha_2}_{23} & \sigma^{\alpha_2}_{24} & \sigma^{\alpha_1}_{41} & \sigma^{\alpha_1}_{42} & \sigma^{\alpha_1}_{43} & \sigma^{\alpha_1}_{44} & \sigma^{\alpha_1}_{45} & \sigma^{\alpha_1}_{46} & \sigma^{\alpha_1}_{11} & \sigma^{\alpha_1}_{12} & \sigma^{\alpha_1}_{13} & \sigma^{\alpha_1}_{14} & \sigma^{\alpha_1}_{15} & \sigma^{\alpha_1}_{16} \\
        
        0 & 0 & 0 & 0 & \sigma^{\alpha_2}_{31} & \sigma^{\alpha_2}_{32} & \sigma^{\alpha_2}_{33} & \sigma^{\alpha_2}_{34} & \sigma^{\alpha_1}_{51} & \sigma^{\alpha_1}_{52} & \sigma^{\alpha_1}_{53} & \sigma^{\alpha_1}_{54} & \sigma^{\alpha_1}_{55} & \sigma^{\alpha_1}_{56} & \sigma^{\alpha_1}_{21} & \sigma^{\alpha_1}_{22} & \sigma^{\alpha_1}_{23} & \sigma^{\alpha_1}_{24} & \sigma^{\alpha_1}_{25} & \sigma^{\alpha_1}_{26} \\
        
        0 & 0 & 0 & 0 & \sigma^{\alpha_2}_{41} & \sigma^{\alpha_2}_{42} & \sigma^{\alpha_2}_{43} & \sigma^{\alpha_2}_{44} & \sigma^{\alpha_1}_{61} & \sigma^{\alpha_1}_{62} & \sigma^{\alpha_1}_{63} & \sigma^{\alpha_1}_{64} & \sigma^{\alpha_1}_{65} & \sigma^{\alpha_1}_{66} & \sigma^{\alpha_1}_{31} & \sigma^{\alpha_1}_{32} & \sigma^{\alpha_1}_{33} & \sigma^{\alpha_1}_{34} & \sigma^{\alpha_1}_{35} & \sigma^{\alpha_1}_{36} \\
        
    \end{pmatrix},$}

\vspace{2mm}
    
    $$A = \begin{pmatrix}
        \bar{\sigma}^{\alpha_2}_{11} & \bar{\sigma}^{\alpha_2}_{12} & \bar{\sigma}^{\alpha_2}_{13} & \bar{\sigma}^{\alpha_2}_{14} & 0 & 0 & 0 & 0 \\
        \bar{\sigma}^{\alpha_2}_{21} & \bar{\sigma}^{\alpha_2}_{22} & \bar{\sigma}^{\alpha_2}_{23} & \bar{\sigma}^{\alpha_2}_{24} & 0 & 0 & 0 & 0 \\ \bar{\sigma}^{\alpha_2}_{31} & \bar{\sigma}^{\alpha_2}_{32} & \bar{\sigma}^{\alpha_2}_{33} & \bar{\sigma}^{\alpha_2}_{34} & 0 & 0 & 0 & 0 \\ \bar{\sigma}^{\alpha_2}_{41} & \bar{\sigma}^{\alpha_2}_{42} & \bar{\sigma}^{\alpha_2}_{43} & \bar{\sigma}^{\alpha_2}_{44} & 0 & 0 & 0 & 0 \\
        0 & 0 & 0 & 0 & \bar{\sigma}^{\alpha_2}_{11} & \bar{\sigma}^{\alpha_2}_{12} & \bar{\sigma}^{\alpha_2}_{13} & \bar{\sigma}^{\alpha_2}_{14} \\
        0 & 0 & 0 & 0 & \bar{\sigma}^{\alpha_2}_{21} & \bar{\sigma}^{\alpha_2}_{22} & \bar{\sigma}^{\alpha_2}_{23} & \bar{\sigma}^{\alpha_2}_{24} \\
        0 & 0 & 0 & 0 & \bar{\sigma}^{\alpha_2}_{31} & \bar{\sigma}^{\alpha_2}_{32} & \bar{\sigma}^{\alpha_2}_{33} & \bar{\sigma}^{\alpha_2}_{34} \\
        0 & 0 & 0 & 0 & \bar{\sigma}^{\alpha_2}_{41} & \bar{\sigma}^{\alpha_2}_{42} & \bar{\sigma}^{\alpha_2}_{43} & \bar{\sigma}^{\alpha_2}_{44} \\
        0 & \bar{\sigma}^{\alpha_1}_{11} & \bar{\sigma}^{\alpha_1}_{12} & \bar{\sigma}^{\alpha_1}_{13} & 0 & \bar{\sigma}^{\alpha_1}_{14} & \bar{\sigma}^{\alpha_1}_{15} & \bar{\sigma}^{\alpha_1}_{16} \\
        0 & \bar{\sigma}^{\alpha_1}_{21} & \bar{\sigma}^{\alpha_1}_{22} & \bar{\sigma}^{\alpha_1}_{23} & 0 & \bar{\sigma}^{\alpha_1}_{24} & \bar{\sigma}^{\alpha_1}_{25} & \bar{\sigma}^{\alpha_1}_{26} \\
        0 & \bar{\sigma}^{\alpha_1}_{31} & \bar{\sigma}^{\alpha_1}_{32} & \bar{\sigma}^{\alpha_1}_{33} & 0 & \bar{\sigma}^{\alpha_1}_{34} & \bar{\sigma}^{\alpha_1}_{35} & \bar{\sigma}^{\alpha_1}_{36} \\
        0 & \bar{\sigma}^{\alpha_1}_{41} & \bar{\sigma}^{\alpha_1}_{42} & \bar{\sigma}^{\alpha_1}_{43} & 0 & \bar{\sigma}^{\alpha_1}_{44} & \bar{\sigma}^{\alpha_1}_{45} & \bar{\sigma}^{\alpha_1}_{46} \\
        0 & \bar{\sigma}^{\alpha_1}_{51} & \bar{\sigma}^{\alpha_1}_{52} & \bar{\sigma}^{\alpha_1}_{53} & 0 & \bar{\sigma}^{\alpha_1}_{54} & \bar{\sigma}^{\alpha_1}_{55} & \bar{\sigma}^{\alpha_1}_{56} \\
        0 & \bar{\sigma}^{\alpha_1}_{61} & \bar{\sigma}^{\alpha_1}_{62} & \bar{\sigma}^{\alpha_1}_{63} & 0 & \bar{\sigma}^{\alpha_1}_{64} & \bar{\sigma}^{\alpha_1}_{65} & \bar{\sigma}^{\alpha_1}_{66} \\
        0 & \bar{\sigma}^{\alpha_1}_{14} & \bar{\sigma}^{\alpha_1}_{15} & \bar{\sigma}^{\alpha_1}_{16} & 0 & \bar{\sigma}^{\alpha_1}_{11} & \bar{\sigma}^{\alpha_1}_{12} & \bar{\sigma}^{\alpha_1}_{13} \\
        0 & \bar{\sigma}^{\alpha_1}_{24} & \bar{\sigma}^{\alpha_1}_{25} & \bar{\sigma}^{\alpha_1}_{26} & 0 & \bar{\sigma}^{\alpha_1}_{21} & \bar{\sigma}^{\alpha_1}_{22} & \bar{\sigma}^{\alpha_1}_{23} \\
        0 & \bar{\sigma}^{\alpha_1}_{34} & \bar{\sigma}^{\alpha_1}_{35} & \bar{\sigma}^{\alpha_1}_{36} & 0 & \bar{\sigma}^{\alpha_1}_{31} & \bar{\sigma}^{\alpha_1}_{32} & \bar{\sigma}^{\alpha_1}_{33} \\
        0 & \bar{\sigma}^{\alpha_1}_{44} & \bar{\sigma}^{\alpha_1}_{45} & \bar{\sigma}^{\alpha_1}_{46} & 0 & \bar{\sigma}^{\alpha_1}_{41} & \bar{\sigma}^{\alpha_1}_{42} & \bar{\sigma}^{\alpha_1}_{43} \\
        0 & \bar{\sigma}^{\alpha_1}_{54} & \bar{\sigma}^{\alpha_1}_{55} & \bar{\sigma}^{\alpha_1}_{56} & 0 & \bar{\sigma}^{\alpha_1}_{51} & \bar{\sigma}^{\alpha_1}_{52} & \bar{\sigma}^{\alpha_1}_{53} \\
        0 & \bar{\sigma}^{\alpha_1}_{64} & \bar{\sigma}^{\alpha_1}_{65} & \bar{\sigma}^{\alpha_1}_{66} & 0 & \bar{\sigma}^{\alpha_1}_{61} & \bar{\sigma}^{\alpha_1}_{62} & \bar{\sigma}^{\alpha_1}_{63} \\
    \end{pmatrix}.$$

\end{example}

\section*{Acknowledgement}
The authors would like to thank Prof. Roozbeh Hazrat (University of Western Sydney) for his insightful contributions during the early stages of this study. They would also like to thank Afilgen Sebandal (Research Center for Theoretical Physics, Jagna, Bohol) for her thorough review of the manuscript. Finally, the second author would like to thank his supervisor Ardeline Mary Buhphang (North-Eastern Hill University) for her willingness to discuss the material presented in this manuscript.

The first author was partially supported by the Australian Government RTP scholarship.

\newpage 
\section*{Declarations}
\subsection*{Ethical approval}
Not applicable

\subsection*{Competing interests}
Not applicable

\subsection*{Authors' contribution}
All authors have contributed equally to all sections.

\subsection*{Funding}
Not applicable

\subsection*{Availability of data and materials}
Not applicable

\printbibliography
\end{document}